\documentclass[a4paper]{article}
\pdfoutput=1
\usepackage{graphicx,wrapfig}
\usepackage{lipsum}
\usepackage{amsmath,amssymb,mathrsfs}
\usepackage{hyperref}
\usepackage{url}
\usepackage{mathtools}
\usepackage{changepage}
\usepackage{multirow}
\usepackage{wrapfig}
\usepackage{subfig}
\usepackage{xcolor}
\usepackage{epsfig,enumerate,amsmath,amsfonts,amssymb,amsthm,mathrsfs,ifpdf}
\usepackage{indentfirst,relsize}
\usepackage{setspace,graphicx}
\usepackage{latexsym}
\usepackage[all]{xy}
\usepackage[usenames,dvipsnames]{pstricks}
\usepackage{pst-grad} % For gradients
\usepackage{pst-plot} % For axes
\usepackage[margin = 2.50cm]{geometry}

\usepackage[linesnumbered,ruled,vlined]{algorithm2e}

\usepackage{color,soul}

\usepackage[colorinlistoftodos]{todonotes}

\usepackage{authblk} % to add author affiliations

\usepackage{orcidlink}% for orcid id

\usepackage{enumitem}% http://ctan.org/pkg/enumitem

\usepackage{tikz}
\usetikzlibrary{calc}

\newcommand{\remove}[1]{}

\newtheorem{theorem}{Theorem}%[section]
\newtheorem{lemma}[theorem]{Lemma}

\newtheorem{corollary}[theorem]{Corollary}

\newtheorem{conjecture}{Conjecture}

\newcounter{Case}[theorem]

\newcounter{Ca}[theorem]%write a case inside a case

\newcounter{NCa}[theorem]

\newtheorem{remark}{Remark}

\usepackage{xcolor}
\allowdisplaybreaks

\title{An Improved Upper Bound for the Strong Odd Chromatic Number of Planar Graphs}
\author[]{Kamal Santra \orcidlink{0009-0006-5997-1452} \footnote{kamal.7.2013@gmail.com, kamal.santra@iitg.ac.in}}
\affil[]{Department of Mathematics\\
	
	Indian Institute of Technology Guwahati\\
	
	Guwahati, 781039, Assam, India}
\date{}

\begin{document}

\maketitle
\begin{abstract}
	A proper coloring of a graph is called a strong odd coloring if, for every
	vertex \(v\) and every color appearing in the open neighborhood of \(v\),
	that color appears an odd number of times in \(N(v)\). The corresponding
	minimum number of colors is the strong odd chromatic number, denoted by
	\(\chi_{\mathrm{so}}(G)\). Caro et
	al.~\cite{CaroPetrusevskiSkrekovskiTuzaStrongOdd} proved that every
	planar graph has strong odd chromatic number at most \(388\). Manattu et
	al.~\cite{ManattuVinayLakshmanan2026} later constructed planar graphs
	with strong odd chromatic number \(17\) and asked whether larger values
	are possible and whether the upper bound \(388\) can be improved. We
	address these questions as follows. First, we improve the general upper
	bound by proving that every planar graph \(G\) satisfies
	\(\chi_{\mathrm{so}}(G)\le 368\). This follows by improving the auxiliary
	proper facially odd coloring bound for loopless \(2\)-connected plane
	multigraphs from \(97\) colors to \(92\) colors and combining this with
	the reduction of Caro et al. and the Four Color Theorem. Second, we
	give a different explicit planar construction with
	\(\chi_{\mathrm{so}}(G)=20\), together with a self-contained proof of the
	exact value. We emphasize that Goetze et
	al.~\cite{GoetzeKluteKnauerParadaPenaUeckerdt2025} had already posted an
	arXiv preprint in May 2025 containing a planar example with strong odd
	chromatic number \(20\). Thus our construction is not a priority claim for
	the value \(20\), but rather an independent and fully verified
	construction whose value exceeds \(17\), the value that motivated
	Problem~1 of Manattu et al.
\end{abstract}


\section{Introduction}
\label{sec:introduction}

Graph colorings with constraints on open neighbourhoods form a natural
family of parameters lying between ordinary proper coloring and distance
coloring. A \(k\)-coloring of a graph \(G\) is a map
\(\varphi:V(G)\to\{1,\ldots,k\}\), and it is proper if adjacent vertices
receive distinct colors. A square coloring, or \(2\)-distance coloring, is
a proper coloring of \(G^2\), where two vertices are adjacent in \(G^2\) whenever they are at a distance at most two in \(G\). Between these two
extremes, several neighbourhood- and uniqueness variants have been
studied in recent years.

Odd coloring was introduced by Petru{\v{s}}evski and
{\v{S}}krekovski~\cite{PetrusevskiSkrekovski2022}. A proper coloring of
\(G\) is an odd coloring if every non-isolated vertex \(v\) has at least
one color that appears an odd number of times in \(N_G(v)\). The minimum
number of colors in such a coloring is the odd chromatic number, denoted
by \(\chi_{\mathrm{o}}(G)\). This parameter already behaves quite
differently from the ordinary chromatic number: the inequality
\(\chi(G)\le \chi_{\mathrm{o}}(G)\) can be strict, and determining the
odd chromatic number is computationally hard in
general~\cite{AhnImOum2022,CaroPetrusevskiSkrekovski2022}. Petru{\v{s}}evski
and {\v{S}}krekovski~\cite{PetrusevskiSkrekovski2022} proved that every
planar graph is odd \(9\)-colorable and conjectured that every planar
graph is odd \(5\)-colorable. Petr and Portier~\cite{PetrPortier2023}
improved the general planar upper bound to \(8\). Further results on odd
colorings of planar graphs, sparse graphs, and related graph classes were
obtained in~\cite{CaroPetrusevskiSkrekovski2022,ChoChoiKwonPark2023,Cranston2024,MiaoSunTuYu2024}.

A stronger variant, called strong odd coloring, was introduced by Kwon and
Park~\cite{KwonPark2024}. A proper coloring \(\varphi\) of \(G\) is a
strong odd coloring if, for every vertex \(v\in V(G)\) and every color
\(\alpha\), the color \(\alpha\) appears either zero times or an odd
number of times in \(N_G(v)\). The minimum number of colors in a strong
odd coloring of \(G\) is denoted by \(\chi_{\mathrm{so}}(G)\). Thus, every
strong odd coloring is an odd coloring, and every square coloring is a
strong odd coloring. In particular, for every graph \(G\), one has
\[
\chi(G)\le \chi_{\mathrm{o}}(G)\le \chi_{\mathrm{so}}(G)\le \chi(G^2).
\]
However, \(\chi_{\mathrm{so}}\) is not simply a square-coloring parameter.
Kwon and Park~\cite{KwonPark2024} observed that the difference
\(\chi(G^2)-\chi_{\mathrm{so}}(G)\) can be arbitrarily large, and they
studied strong odd colorings of sparse graphs. They also asked whether
there exists an absolute constant \(C\) such that
\(\chi_{\mathrm{so}}(G)\le C\) for every planar graph \(G\).

Caro et al.~\cite{CaroPetrusevskiSkrekovskiTuzaStrongOdd} answered this
boundedness question affirmatively. Their approach relates strong odd
colorings of planar graphs to proper facially odd colorings of plane
multigraphs. Let \(\mathcal P\) be the class of loopless \(2\)-connected
plane multigraphs, and let \(\chi_{\mathrm{pfo}}(\mathcal P)\) denote
the maximum proper facially odd chromatic number over this class. They
proved that every planar graph \(G\) satisfies
\(\chi_{\mathrm{so}}(G)\le \chi(G)\chi_{\mathrm{pfo}}(\mathcal P)\). The
known bound \(\chi_{\mathrm{pfo}}(\mathcal P)\le 97\), due to Kaiser et al.~\cite{KaiserRuckyStehlikSkrekovski2014},
together with the Four Color Theorem~\cite{RobertsonSandersSeymourThomas1996},
gives \(\chi_{\mathrm{so}}(G)\le 4\cdot97=388\) for every planar graph
\(G\). The same reduction, together with the tight proper facially odd
bound for outerplane graphs due to Wang et al.~\cite{WangFinbowWang2012}, gives the earlier upper bound \(30\) for
outerplanar graphs. Goetze et
al.~\cite{GoetzeKluteKnauerParadaPenaUeckerdt2025} subsequently reduced
the latter bound to \(8\).

The lower-bound side has also developed rapidly. Caro et
al.~\cite{CaroPetrusevskiSkrekovskiTuzaStrongOdd} constructed two planar
graphs with strong odd chromatic number \(12\) and asked whether every
planar graph satisfies \(\chi_{\mathrm{so}}(G)\le 12\). Pang et al.~\cite{PangMiaoFan2026} answered this question negatively by
constructing a planar graph with strong odd chromatic number \(13\), and
they conjectured that every planar graph is strongly odd \(13\)-colorable.
Manattu et al.~\cite{ManattuVinayLakshmanan2026}
disproved that conjecture by constructing infinitely many planar graphs
with strong odd chromatic number \(14\), and also gave examples with
values \(15\), \(16\), and \(17\). Their preprint was posted in February
2026 and ends with three problems, including the problem of identifying
planar graphs with strong odd chromatic number greater than \(17\), and
the problem of reducing the upper bound \(388\).

There is an important chronological point concerning this first problem.
Before the Manattu--Vinay--Lakshmanan preprint was posted, Goetze et al~\cite{GoetzeKluteKnauerParadaPenaUeckerdt2025}
had already posted the arXiv preprint \emph{Strong odd coloring in
	minor-closed classes} in May 2025. In that preprint, they exhibit a
planar graph \(G''\) with \(\chi_{\mathrm{so}}(G'')=20\). They also
write that the analysis of the example is left to the reader. Thus, the
existence of a planar graph with strong odd chromatic number \(20\) was
already available before the later journal submission containing
Problem~1. In the present paper, we therefore do not claim priority for
the value \(20\). Instead, we provide a different explicit construction
and a complete self-contained proof of the exact value.

The main contribution of the present paper is an improved universal upper
bound for planar graphs. We refine the proper facially odd coloring
ingredient used by Caro et al. More precisely, we improve the auxiliary
bound from \(\chi_{\mathrm{pfo}}(\mathcal P)\le 97\) to
\(\chi_{\mathrm{pfo}}(\mathcal P)\le 92\). Combining this with the
reduction of Caro et al. and the Four Color Theorem gives the following
theorem.

\begin{theorem}
	\label{thm:intro-368}
	Every planar graph \(G\) satisfies \(\chi_{\mathrm{so}}(G)\le 368\).
\end{theorem}

This answers Problem~2 of Manattu et al.~\cite{ManattuVinayLakshmanan2026}; in the strict formulation
asking for a constant \(c<388\) such that \(\chi_{\mathrm{so}}(G)<c\)
for every planar graph \(G\), one may take \(c=369\).

We also include an explicit planar construction with strong odd chromatic
number \(20\). As explained above, the value \(20\) was already announced
by Goetze et al.~\cite{GoetzeKluteKnauerParadaPenaUeckerdt2025}. Our
construction is different and short, and we include the full proof. The
graph is built from three paths \(P_{12},P_{13},P_{23}\), each on six
vertices, and three additional vertices \(x_1,x_2,x_3\). The vertices
\(x_1\) and \(x_2\) are adjacent to all vertices of \(P_{12}\), the
vertices \(x_1\) and \(x_3\) are adjacent to all vertices of \(P_{13}\),
and the vertices \(x_2\) and \(x_3\) are adjacent to all vertices of
\(P_{23}\); finally, the edge \(x_1x_2\) is added. We prove that all
\(18\) path vertices must receive pairwise distinct colors, while the
colors of \(x_1\) and \(x_2\) cannot be reused on any path vertex. This
gives the following exact value.

\begin{theorem}
	\label{thm:intro-G20}
	There exists a planar graph \(G\) such that
	\(\chi_{\mathrm{so}}(G)=20\).
\end{theorem}

The paper is organized as follows. Section~\ref{sec:preliminaries} gives
the necessary definitions and notation. Section~\ref{sec:strong-odd-368}
proves the auxiliary bound \(\chi_{\mathrm{pfo}}(\mathcal P)\le 92\) and
derives the planar bound \(368\). Section~\ref{sec:planar-strong-odd-20}
presents our explicit planar graph with a strong odd chromatic number
\(20\), together with a complete proof of the exact value. We conclude
with remarks on the remaining gap between the lower bound \(20\) and the
upper bound \(368\), and explain why the previously known universal bound
already resolves the finiteness question of Manattu et al.

\section{Preliminaries}
\label{sec:preliminaries}

All graphs considered in this paper are finite and simple, unless explicitly stated otherwise. For a graph \(G\), we write \(V(G)\) and \(E(G)\) for its vertex set and edge set. For a vertex \(v\in V(G)\), the open neighbourhood of \(v\) is denoted by \(N_G(v)\), and its degree is \(d_G(v)=|N_G(v)|\). When the graph is clear from the context, we simply write \(N(v)\) and \(d(v)\). A proper coloring of \(G\) is a map \(\varphi:V(G)\to C\), where \(C\) is a set of colors, such that adjacent vertices receive distinct colors. The chromatic number of \(G\) is denoted by \(\chi(G)\).

A proper coloring \(\varphi\) of \(G\) is called a \emph{strong odd coloring} if, for every vertex \(v\in V(G)\) and every color \(\alpha\), the set \(N_G(v)\cap \varphi^{-1}(\alpha)\) has either size zero or odd size. Equivalently, every color that appears in the open neighbourhood of \(v\) appears there an odd number of times. The \emph{strong odd chromatic number} of \(G\), denoted by \(\chi_{\mathrm{so}}(G)\), is the minimum number of colors in a strong odd coloring of \(G\).

We shall use the following immediate observation several times. If a vertex \(v\) has two neighbours of the same color \(\alpha\), and no other neighbour of \(v\) has color \(\alpha\), then the coloring is not strong odd. Thus, in a strong odd coloring, no color may appear exactly twice in the neighbourhood of any vertex.

For two vertex-disjoint graphs \(G_1\) and \(G_2\), their join \(G_1\vee G_2\) is obtained from their disjoint union by adding all edges between \(V(G_1)\) and \(V(G_2)\). We write \(\overline{K_2}\) for the empty graph on two vertices. Thus \(P_6\vee\overline{K_2}\) is the graph obtained from a path on six vertices by adding two independent vertices, each adjacent to every vertex of the path.

We also need a facial coloring notion for plane multigraphs. A plane
multigraph is a multigraph together with a fixed embedding in the plane.
We only use loopless \(2\)-connected plane multigraphs. Let \(\mathcal P\)
denote the class of all such multigraphs. If \(H\in\mathcal P\), a proper
vertex coloring of \(H\) is called a \emph{proper facially odd coloring}
if, on the boundary walk of every face \(f\) and for every color
\(\alpha\), the number of occurrences of vertices colored \(\alpha\) is
either zero or odd. Since every facial boundary of a \(2\)-connected
loopless plane multigraph is a cycle or a digon, this is equivalent to
counting the vertices incident with \(f\). The maximum proper facially odd
chromatic number over all graphs in \(\mathcal P\) is denoted by
\(\chi_{\mathrm{pfo}}(\mathcal P)\).

For a positive integer \(k\), a \(k\)-vertex means a vertex of degree exactly
\(k\), a \(k^{-}\)-vertex means a vertex of degree at most \(k\), and a
\(k^{+}\)-vertex means a vertex of degree at least \(k\).

For a vertex \(v\) of a plane multigraph \(H\), let \(N_F(v)\) be the set of
vertices different from \(v\) that appear together with \(v\) on some common
face. The facial degree of \(v\) is \(d_F(v)=|N_F(v)|\). If \(f\) is a face
incident with \(v\), then \(d_F(v,f)\) denotes the number of vertices different
from \(v\) that appear together with \(v\) on a face different from \(f\). For a
face \(f\), let \(w(f)\) denote the number of \(3^{+}\)-vertices incident with
\(f\). A face of weight \(2\) which is not a digon is called a pseudodigon. In
the proof of the upper bound, we use the modified weight \(w'(f)\), defined by
\(w'(f)=3\) if \(f\) is a pseudodigon and \(w'(f)=w(f)\) otherwise.

We shall use the following reduction of Caro et
al.~\cite{CaroPetrusevskiSkrekovskiTuzaStrongOdd}. For every planar graph
\(G\), one has \(\chi_{\mathrm{so}}(G)\le
\chi(G)\chi_{\mathrm{pfo}}(\mathcal P)\). Since every planar graph is
\(4\)-colorable by the Four Color
Theorem~\cite{RobertsonSandersSeymourThomas1996}, any bound
\(\chi_{\mathrm{pfo}}(\mathcal P)\le q\) implies
\(\chi_{\mathrm{so}}(G)\le 4q\) for every planar graph \(G\).

\section{A bound of 368 for strong odd colorings of planar graphs}
\label{sec:strong-odd-368}

By the reduction recalled in Section~\ref{sec:preliminaries} and the Four
Color Theorem, it is enough to prove
\(\chi_{\mathrm{pfo}}(\mathcal P)\le92\). We establish this auxiliary
bound first and then state its consequence for planar graphs.

\subsection{The auxiliary proper facially odd coloring bound}

\begin{theorem}
	\label{thm:pfo-92}
	Every loopless \(2\)-connected plane multigraph admits a proper facially odd coloring with at most \(92\) colors. Equivalently, \(\chi_{\mathrm{pfo}}(\mathcal P)\le 92\).
\end{theorem}

\begin{proof}
	Suppose, to the contrary, that \(H\in\mathcal P\) is a counterexample with the minimum number of vertices and, subject to this, the minimum number of edges. We derive a contradiction by adapting the facial-degree and discharging framework used for the known \(97\)-color bound.
	
	\subsubsection{Consequences of minimality}
	
	For a vertex \(v\), let \(\mathcal F(v)\) be the set of faces incident
	with \(v\). Recall that \(N_F(v)\) is the union of the boundary
	vertices of the faces in \(\mathcal F(v)\), with \(v\) itself removed,
	and \(d_F(v)=|N_F(v)|\). If \(f\in\mathcal F(v)\), then
	\(N_F(v,f)\) and \(d_F(v,f)\) are defined in the same way after
	omitting \(f\).
	
	We need the following consequences of minimality. They are the
	\(92\)-color analogues of the reducibility properties in the proof of
	Kaiser et
	al.~\cite{KaiserRuckyStehlikSkrekovski2014}.
	
	\medskip
	\noindent\emph{Claim 1.} The graph \(H\) has the following properties:
	\begin{enumerate}
		\item[(i)] \(|V(H)|>92\);
		\item[(ii)] \(H\) has no parallel edges and hence no digons;
		\item[(iii)] no facial boundary contains four consecutive
		\(2\)-vertices;
		\item[(iv)] \(d_F(v)\ge 92\) for every \(v\in V(H)\);
		\item[(v)] if \(\mathcal F(u)\cap\mathcal F(v)=\{f\}\), then
		\(d_F(u,f)+d_F(v,f)\ge 91\).
	\end{enumerate}
	
	\noindent\emph{Proof of Claim 1.}
	If \(|V(H)|\le 92\), assigning a different color to every vertex is
	a proper facially odd coloring, which proves (i).
	
	For (ii), suppose that two parallel edges form a digon. Deleting one
	of them gives a smaller loopless \(2\)-connected plane multigraph.
	A proper facially odd coloring of the smaller graph is also valid for
	\(H\), since properness makes the two boundary colors of the restored
	digon distinct. If the two parallel edges do not bound a digon, their
	union is a Jordan curve. The subgraphs drawn on its two sides are
	smaller \(2\)-connected plane multigraphs. After permuting color
	names, their \(92\)-colorings agree on the two common endvertices and
	combine into a \(92\)-coloring of \(H\), again a contradiction.
	
	To prove (iii), suppose that \(x_1x_2x_3x_4\) occurs consecutively on
	a facial boundary and all four vertices have degree \(2\). Let \(y_1\)
	and \(y_4\) be the other neighbours of \(x_1\) and \(x_4\),
	respectively. The vertices \(y_1,y_4\) are distinct from each other and
	from all four \(x_i\): otherwise that facial boundary would close into
	a cycle on four or five vertices whose only possible attachment to the
	rest of \(H\) is one vertex. By \(2\)-connectivity there is no such
	attachment, contrary to (i). Contract the path
	\(x_1x_2x_3x_4y_4\) to one vertex. No loop is created, because the
	\(x_i\) have degree \(2\) and \(y_1\ne y_4\). To see that the quotient
	is \(2\)-connected, first delete its contracted vertex. The graph that
	remains is \(H-\{x_1,x_2,x_3,x_4,y_4\}\), which is connected because
	\(H-y_4\) is connected and \(x_4x_3x_2x_1\) is a pendant path there,
	attached only at \(y_1\). If instead any other vertex \(r\) is deleted,
	then \(H-r\) is connected and contracting the still-connected path
	\(x_1x_2x_3x_4y_4\) preserves connectedness. Thus the quotient has no
	cutvertex; it also has more than two vertices by (i). In a
	\(92\)-coloring of the contracted plane multigraph,
	the colors on \(y_1\) and on the contracted vertex are distinct.
	Alternating these two colors along
	\(y_1x_1x_2x_3x_4y_4\) restores a proper coloring and adds each color
	an even number of times to the affected facial boundaries. Thus all
	facial parities are preserved, a contradiction.
	
	We next recall the annihilation operation. Let \(z\) have cyclically
	ordered, pairwise distinct neighbors \(z_0,\ldots,z_{d-1}\). Delete
	\(z\), and for every \(i\) add the edge \(z_i z_{i+1}\) inside the
	angle formerly bounded by \(zz_i\) and \(zz_{i+1}\), with indices
	taken modulo \(d\). When the original graph has at least four
	vertices and no parallel edge at \(z\), the resulting plane
	multigraph is loopless and \(2\)-connected: every old facial boundary
	remains a cycle after replacing the two-edge segment through \(z\) by
	the new edge, and the additional boundary
	\(z_0z_1\cdots z_{d-1}z_0\) is a cycle, or a digon when \(d=2\).
	Thus every facial boundary of the resulting graph is a cycle or digon,
	which is the facial-boundary criterion for \(2\)-connectivity.
	
	If \(d_F(z)\le 91\), annihilate \(z\) and color the smaller graph.
	At least one of the \(92\) colors is absent from \(N_F(z)\). Giving
	\(z\) such a color preserves properness and adds that color exactly
	once to every facial boundary incident with \(z\). This contradicts
	the choice of \(H\) and proves (iv).
	
	Finally, suppose that
	\(\mathcal F(u)\cap\mathcal F(v)=\{f\}\) and
	\(d_F(u,f)+d_F(v,f)\le 90\). The vertices \(u\) and \(v\) are
	nonadjacent, since every edge of a \(2\)-connected plane graph is
	incident with two faces. Annihilate \(u\) and then \(v\). Since they
	are nonadjacent, the first annihilation creates no edge incident with
	\(v\), so the second annihilation again satisfies the preceding
	hypotheses. Take a \(92\)-coloring of the resulting graph and put
	\(A=N_F(u,f)\cup N_F(v,f)\). At least two colors are absent from
	\(A\). If one such color occurs on a boundary vertex of \(f\) outside
	\(A\), give that color to both \(u\) and \(v\); its parity on \(f\)
	is unchanged because two occurrences are restored. Otherwise, give
	\(u\) and \(v\) two distinct colors absent from \(A\); these colors
	are then absent from the remainder of \(f\) and each is restored
	once. In both cases, all other affected faces gain one occurrence of
	a previously absent color. The extension is proper and facially odd,
	a contradiction. This proves (v) and completes Claim 1.
	
	\subsubsection{Modified face weights}
	
	For a face \(f\), let \(w(f)\) be the number of \(3^+\)-vertices on
	its boundary. A face of weight \(2\) that is not a digon is called a
	pseudodigon. Define
	\[
	w'(f)=
	\begin{cases}
		3,&\text{if \(f\) is a pseudodigon},\\
		w(f),&\text{otherwise}.
	\end{cases}
	\]
	The graph \(H\) is not a cycle: by Claim~1(i) such a cycle would have
	more than \(92\) vertices, contradicting Claim~1(iii). Moreover, in a
	\(2\)-connected graph that is not a cycle, every facial boundary has
	at least two \(3^+\)-vertices. Consequently, \(w'(f)\ge 3\) for every
	face \(f\).
	
	The next claim supplies all facial-degree estimates used below. It is
	a refinement, with the cutoff \(24\), of the counting lemma of Kaiser
	et al.~\cite{KaiserRuckyStehlikSkrekovski2014}.
	
	\medskip
	\noindent\emph{Claim 2.} Let \(v\) be a vertex of degree \(d\), let
	\(f_0,\ldots,f_{d-1}\) be its incident faces in cyclic order, and put
	\(l_i=w'(f_i)\).
	\begin{enumerate}
		\item[(a)] If \(d\ge 3\), then
		\begin{equation}
			d_F(v)\le
			4\sum_{i=0}^{d-1}l_i-5d-3\sigma,
			\label{eq:facial-degree}
		\end{equation}
		where
		\(\sigma=|\{i:l_i+l_{i+1}\le 24\}|\), with indices modulo \(d\).
		\item[(b)] If \(d=2\), then
		\begin{equation}
			d_F(v)\le 4\bigl(w(f_0)+w(f_1)\bigr)-6.
			\label{eq:degree-two}
		\end{equation}
		\item[(c)] If \(d\ge 3\), then
		\begin{equation}
			d_F(v,f_0)\le
			4\sum_{i=1}^{d-1}l_i-5d+9.
			\label{eq:reduced-facial-degree}
		\end{equation}
		If, in addition, \(d=3\) and \(l_1+l_2\le 24\), then
		\begin{equation}
			d_F(v,f_0)\le 4(l_1+l_2)-9.
			\label{eq:sharp-reduced-facial-degree}
		\end{equation}
	\end{enumerate}
	
	\noindent\emph{Proof of Claim 2.}
	A \emph{thread} is a maximal path whose internal vertices have degree
	\(2\) and whose endvertices have degree at least \(3\). Such threads
	cover all edges of \(H\), because \(H\) is not a cycle. By
	Claim~1(iii), every thread has at most three internal vertices.
	
	We first prove part (b). Let a \(2\)-vertex \(u\) lie on a thread
	\(P\) with high-degree endvertices, and let its two incident faces have
	weights \(a=w(f_0)\) and \(b=w(f_1)\). After all threads are suppressed
	to edges, the two facial boundaries have \(a\) and \(b\) high-degree
	vertices, respectively. Both endvertices of \(P\) occur on both
	boundaries, so their union has at most \(a+b-2\) high-degree vertices.
	After the shared thread \(P\) is omitted, the two boundaries likewise
	use at most \(a+b-2\) other threads.
	Thus \(N_F(u)\) contains at most \(a+b-2\) high-degree vertices, at
	most \(3(a+b-2)\) internal vertices of those other threads, and at
	most two further internal vertices of \(P\). Hence,
	\[
	d_F(u)\le(a+b-2)+3(a+b-2)+2=4(a+b)-6,
	\]
	which is \eqref{eq:degree-two}.
	
	Now let \(d\ge3\). Denote the threads incident with \(v\), in cyclic
	order, by \(P_0,\ldots,P_{d-1}\), so that \(f_i\) lies between
	\(P_i\) and \(P_{i+1}\). Traverse the boundary of \(f_i\) from \(v\)
	along \(P_i\), and let \(Q_i\) be the boundary segment that starts at
	the first vertex after \(v\) and stops immediately before the last
	high-degree vertex preceding the return along \(P_{i+1}\). If \(f_i\)
	is a pseudodigon, enlarge \(Q_i\) by that last high-degree vertex,
	which is its unique high-degree boundary vertex other than \(v\).
	
	If \(f_i\) is not a pseudodigon, \(Q_i\) contains at most
	\(w(f_i)-2\) high-degree vertices and the internal vertices of at most
	\(w(f_i)-1\) threads. If \(f_i\) is a pseudodigon, the enlarged
	\(Q_i\) contains at most its other high-degree vertex and three
	\(2\)-vertices. Consequently, in both cases,
	\(|V(Q_i)|\le4l_i-5\).
	
	These enlarged segments cover \(N_F(v)\): the part omitted at the end
	of \(Q_i\) lies on \(P_{i+1}\) and is covered from its other direction
	by \(Q_{i+1}\); in the pseudodigon case the shared high-degree
	endvertex was included explicitly.
	
	If \(P_i\) is a single edge, then the first of the thread intervals
	allowed in the preceding count has no internal vertex. The bound for
	\(Q_i\) therefore improves to
	\(
	|V(Q_i)|\le4l_i-8.
	\)
	(The same inequality holds automatically for a pseudodigon, where the
	enlarged segment has at most four vertices.) Suppose now that
	\(l_i+l_{i+1}\le24\). If \(P_{i+1}\) had an internal \(2\)-vertex
	\(u\), its incident faces would be \(f_i,f_{i+1}\), and part (b) would
	give
	\[
	d_F(u)\le4\bigl(w(f_i)+w(f_{i+1})\bigr)-6
	\le4\cdot24-6=90,
	\]
	contrary to Claim~1(iv). Thus \(P_{i+1}\) is a single edge, and the
	bound for the distinct segment \(Q_{i+1}\) saves \(3\). Summing over
	all \(i\) proves \eqref{eq:facial-degree}.
	
	Finally, \(Q_1,\ldots,Q_{d-1}\) cover \(N_F(v,f_0)\) except possibly
	for the other endvertex of one incident thread, and its at most three
	internal \(2\)-vertices. Adding these four possible vertices to the
	basic bounds give
	\[
	\sum_{i=1}^{d-1}(4l_i-5)+4
	=4\sum_{i=1}^{d-1}l_i-5d+9,
	\]
	which proves \eqref{eq:reduced-facial-degree}. When \(d=3\) and
	\(l_1+l_2\le24\), the common thread \(P_2\) is a single edge, so
	\(Q_2\) saves \(3\); the last display becomes
	\(4(l_1+l_2)-9\), proving
	\eqref{eq:sharp-reduced-facial-degree}. This completes Claim 2.
	
	We also record two elementary plane-separation facts used in the
	final discharging cases.
	
	\medskip
	\noindent\emph{Claim 3.}
	The following two elementary facts about intersections of facial boundaries
	will be used in the final part of the discharging argument.
	\begin{enumerate}
		\item[(a)] If \(f\) and \(g\) are distinct faces and \(w(g)\le 3\), then
		the intersection of their boundary cycles is a union of disjoint paths, at
		most one of which is nontrivial. Every internal vertex of these paths has
		degree \(2\). In particular, at most two degree-\(3\) vertices are incident
		with both \(f\) and \(g\), and if there are two, they are the endvertices of
		the nontrivial common path.
		
		\item[(b)] Let \(S\) be a set of at least four vertices on the boundary of a
		face \(f\). If every two vertices of \(S\) are incident with a common face
		different from \(f\), then all vertices of \(S\) are incident with one
		common face different from \(f\).
	\end{enumerate}
	
	\noindent\emph{Proof of Claim 3.}
	For part (a), let \(J\) be the intersection of the two boundary
	cycles. Its components are isolated vertices and paths. It has no
	cycle component: such a component would be both boundary cycles and
	would force \(H\) itself to be a cycle. At an internal vertex of a
	common path, the two common boundary edges occupy both facial sides;
	an additional edge would split one of the two faces there. Hence every
	such internal vertex has degree \(2\). Conversely, an endvertex of a
	maximal nontrivial common path has degree at least \(3\), since at a
	\(2\)-vertex the same two boundary edges would continue the path.
	Every nontrivial component of \(J\) therefore, uses two of the at most
	three high-degree boundary vertices of \(g\), so there is at most one.
	Moreover, at a
	common vertex of degree \(3\), the two facial sectors corresponding to
	\(f\) and \(g\) share an incident edge in the cyclic rotation. Such a
	vertex is therefore an endvertex of a nontrivial common boundary path.
	Since there is at most one such path, there are at most two common
	degree-\(3\) vertices, and if there are two, they are its endvertices.
	
	For (b), choose \(u,x,v\in S\) as three consecutive members of \(S\)
	in the cyclic order along the boundary of \(f\), and let
	\(y\in S\setminus\{u,x,v\}\). Draw a simple arc from
	\(u\) to \(v\) through the interior of a face \(g\ne f\) incident
	with both, and a second simple arc from \(v\) to \(u\) through the
	interior of \(f\). The arcs may be chosen so that their union is a
	Jordan curve. By the cyclic order \(u,x,v,y\) on the boundary of
	\(f\), the vertices \(x\) and \(y\) lie on opposite sides of this
	curve. An arc joining \(x\) to \(y\) through a common face different
	from \(f\) must therefore meet the first arc. Since the interiors of
	distinct faces are disjoint, that common face is \(g\). Varying
	\(y\) shows that every vertex of \(S\) is incident with \(g\),
	proving the claim.
	
	\subsubsection{Discharging rules}
	
	For a face \(f\), let \(|f|\) denote the number of vertices on its
	boundary cycle. Assign initial charge
	\(\operatorname{ch}_0(v)=d(v)-6\) to each vertex and
	\(\operatorname{ch}_0(f)=2|f|-6\) to each face. Euler's formula gives
	\begin{equation}
		\sum_{v\in V(H)}\operatorname{ch}_0(v)
		+\sum_{f\in F(H)}\operatorname{ch}_0(f)=-12.
		\label{eq:total-charge}
	\end{equation}
	
	We use two initial redistribution rules.
	\begin{enumerate}
		\item[(D1)] Every non-pseudodigon sends \(2\) units to each
		incident \(2\)-vertex. A pseudodigon sends \(2\) units to all but
		one of its incident \(2\)-vertices.
		\item[(D2)] A face is \emph{small} if \(w'(f)\le16\),
		\emph{medium} if \(17\le w'(f)\le19\), and \emph{large} if
		\(w'(f)\ge20\). A small face distributes all its remaining
		charge equally among its incident \(3^+\)-vertices. A medium
		face retains \(w'(f)/2-6\) and distributes the rest equally. A
		large face retains \(4\) and distributes the rest equally.
	\end{enumerate}
	
	After (D1), a face \(f\) has charge \(2w'(f)-6\). Therefore a
	\(3^+\)-vertex receives \(2-6/w'(f)\) from an incident small face,
	exactly \(3/2\) from an incident medium face, and
	\(2-10/w'(f)\ge3/2\) from an incident large face. A vertex is
	\emph{special} if its charge is negative after (D1)--(D2), and its
	\emph{deficit} is the absolute value of this negative charge.
	
	\subsubsection{Special vertices are incident with non-small faces}
	
	We first prove that every special vertex is incident with a non-small
	face. Suppose that a \(3^+\)-vertex \(v\) is incident only with small
	faces. No vertex of degree at least \(6\) is special, since both its
	initial charge and all received charges are nonnegative. Thus
	\(d=d(v)\in\{3,4,5\}\). If \(l_1,\ldots,l_d\le16\) are the modified
	weights around \(v\), negativity is equivalent to
	\begin{equation}
		\sum_{i=1}^{d}\frac1{l_i}>\frac d2-1.
		\label{eq:small-reciprocal}
	\end{equation}
	
	Order the \(l_i\)'s increasingly. The short case analysis behind the
	facial-degree maxima is as follows. For \(d=3\), if \(l_1=3\), then
	\(1/l_2+1/l_3>1/6\). When \(l_2\le9\) this gives
	\(l_1+l_2+l_3\le28\), while for \(l_2\ge10\) it gives
	\(l_2+l_3\le24\); in either case the two pairs containing \(l_1\)
	have sum at most \(24\). If \(l_1\ge4\), then \(l_1\le5\), and the
	same reciprocal inequality gives total sum at most \(25\), with all
	three cyclic pair sums at most \(24\). For \(d=4\), three entries
	equal to \(3\) allow the fourth to be at most \(16\). With exactly two
	entries equal to \(3\), the other entries are at most \((4,11)\) or
	\((5,7)\), so their total sum is at most \(21\); with fewer entries
	equal to \(3\), the total is smaller. Finally, for \(d=5\), at most one
	entry can exceed \(3\), and that entry is at most \(5\). Substitution
	in \eqref{eq:facial-degree} gives
	\[
	\begin{array}{c|c|c}
		d & \text{extremal tuple} & \text{bound from
			\eqref{eq:facial-degree}}\\ \hline
		3 & (3,9,16) & 91\\
		4 & (3,3,3,16) & 68\\
		5 & (3,3,3,3,5) & 28
	\end{array}
	\]
	In the first row, at least two cyclic pairs have a sum at most \(24\);
	in the last two rows every cyclic pair does. Thus the displayed
	bounds follow from \eqref{eq:facial-degree}, and each contradicts
	Claim~1(iv).
	
	If \(v\) is a special \(2\)-vertex, one of its incident faces must
	be a pseudodigon whose exceptional \(2\)-vertex in (D1) is \(v\).
	If its other incident face were small, then
	\eqref{eq:degree-two} would give
	\(d_F(v)\le4(2+16)-6=66<92\). Hence, every special vertex is
	incident with a non-small face.
	
	\subsubsection{Reduced configurations of special vertices}
	
	Every non-small face gives at least \(3/2\) to each incident
	\(3^+\)-vertex. Consequently, a special \(3^+\)-vertex has exactly
	one incident non-small face and has a degree at most \(4\). Let \(f\)
	be this unique face.
	
	For a degree-\(3\) vertex, write its two remaining modified weights
	as \(a\le b\). Using \(3/2\) as a lower bound for the contribution
	from \(f\), negativity implies
	\(1/a+1/b>5/12\). Hence either \(a=3\) and \(b\le11\), or
	\(a=4\) and \(b\le5\). For a degree-\(4\) vertex, negativity implies
	that the sum of the reciprocals of the three remaining weights is
	greater than \(11/12\), which forces all three weights to be \(3\).
	A special \(2\)-vertex has the actual \(f\)-reduced configuration
	\((2)\), corresponding to its pseudodigon. Thus the possibilities
	and their charges are
	\[
	\begin{array}{c|c|c}
		d(v) & \text{reduced configuration} & \text{charge at least}\\ \hline
		2 & (2) & -2\\
		3 & (3,x),\ x\le11 & 1/2-6/x\\
		3 & (4,x),\ x\le5 & 1-6/x\\
		4 & (3,3,3) & -1/2
	\end{array}
	\]
	The first row uses the actual weight \(2\); the other rows use
	modified weights.
	
	Equations \eqref{eq:reduced-facial-degree} and
	\eqref{eq:sharp-reduced-facial-degree} give
	\[
	\begin{array}{c|c}
		\text{reduced configuration} & d_F(v,f)\text{ is at most}\\ \hline
		(3,11) & 47\\
		(3,x),\ x\le10 & 43\\
		(4,x),\ x\le5 & 27\\
		(3,3,3) & 25\\
		(2) & 7
	\end{array}
	\]
	For the last row, the other face is a pseudodigon. Its two
	\(3^+\)-vertices are separated by at most three \(2\)-vertices on
	each boundary arc, so it has at most eight boundary vertices and
	\(d_F(v,f)\le7\).
	
	Call a special vertex \emph{exceptional on \(f\)} if its reduced
	configuration is \((3,11)\). If two special vertices on \(f\) have
	no other common incident face, both must be exceptional. Indeed, if
	at least one is nonexceptional; their reduced facial degrees sum to
	at most \(43+47=90\), contrary to Claim~1(v).
	
	\subsubsection{Medium faces}
	
	Let \(f\) be a medium face and put
	\(t=w'(f)\in\{17,18,19\}\). A special \(2\)-vertex cannot be incident
	with \(f\), because \eqref{eq:degree-two} would give
	\(d_F(v)\le4(t+2)-6\le78\). A reduced configuration \((4,x)\) gives
	the largest full facial-degree bound at \((4,5,19)\), namely \(88\);
	a reduced configuration \((3,3,3)\) gives a bound at most \(80\).
	Both contradict Claim~1(iv). Hence, every special vertex on \(f\)
	has reduced configuration \((3,x)\).
	
	Substituting the full configuration \((3,x,t)\) into
	\eqref{eq:facial-degree}, Claim~1(iv) forces the following lower
	bounds. Since a medium face contributes exactly \(3/2\), the deficit
	of a nonexceptional vertex is at most \(6/x-1/2\).
	\[
	\begin{array}{c|c|c|c}
		t & x\text{ is at least} & \text{largest deficit}
		& \text{charge retained by }f\\ \hline
		17 & 9 & 1/6 & 5/2\\
		18 & 8 & 1/4 & 3\\
		19 & 7 & 5/14 & 7/2
	\end{array}
	\]
	For instance, when \(t=17\), the value \(x=8\) gives
	\(d_F(v)\le91\); the other two rows follow identically. An
	exceptional vertex has deficit
	\(6/11-1/2=1/22\).
	
	Let \(C\) be the set of nonexceptional special vertices on \(f\) and
	put \(k=|C|\). Every vertex in \(C\) shares a second face with every
	special vertex on \(f\), by the last paragraph of the preceding
	subsection. If \(k\ge4\), Claim~3(b) gives a face \(g\ne f\)
	incident with every vertex of \(C\). Thus, the second entry \(x\) of
	each reduced configuration in \(C\) is at least \(k\).
	
	Let \(b_{17}=9\), \(b_{18}=8\), and \(b_{19}=7\). There are at most
	\(t-k\) exceptional vertices, so the total deficit is at most
	\[
	k\left(\frac{6}{\max\{b_t,k\}}-\frac12\right)
	+\frac{t-k}{22}
	\qquad (k\ge4).
	\]
	The same formula with denominator \(b_t\) applies for \(k\le3\).
	For \(0\le k\le b_t\) this upper bound is increasing in \(k\). For
	\(k\ge b_t\) it equals
	\(6-k/2+(t-k)/22\), which is decreasing. Its maximum therefore, occurs
	at \(k=b_t\), giving
	\[
	\begin{array}{c|c|c}
		t & \text{largest total deficit} & \text{retained charge}\\ \hline
		17 & 41/22 & 5/2\\
		18 & 27/11 & 3\\
		19 & 67/22 & 7/2
	\end{array}
	\]
	The maxima occur at \(k=9,8,7\), respectively. In every row, the
	retained charge is larger than the total deficit. Thus, every medium
	face can pay the deficits of all its incident special vertices.
	
	\subsubsection{Large faces of modified weight 20 or 21}
	
	Let \(t=w'(f)\in\{20,21\}\). A special \(2\)-vertex is impossible,
	since \eqref{eq:degree-two} gives
	\(d_F(v)\le4(t+2)-6\le86\). A reduced configuration
	\((3,3,3)\) is also excluded by
	\eqref{eq:facial-degree}. For a reduced configuration \((3,x)\),
	Claim~1(iv) forces \(x\ge6\) when \(t=20\), and \(x\ge5\) when
	\(t=21\). The remaining reduced configurations of type \((4,x)\)
	are \((4,5)\) for \(t=20\), and \((4,4)\) or \((4,5)\) for
	\(t=21\); their deficits are no larger than the maximum deficit of
	a nonexceptional \((3,x)\)-vertex.
	
	For a large face of weight \(t\), the deficits of the three
	high-degree types are, respectively,
	\[
	\frac{10}{t}+\frac6x-1,\qquad
	\frac{10}{t}+\frac6x-\frac32,\qquad
	\frac{10}{t}.
	\]
	An exceptional vertex therefore has deficit \(1/22\) for \(t=20\)
	and \(5/231\) for \(t=21\).
	
	Let \(C\) again be the nonexceptional special vertices and
	\(k=|C|\). Every vertex of \(C\) shares another face with every
	special vertex. If \(k\ge4\), Claim~3(b) shows that the largest
	reduced weight of each vertex in \(C\) is at least \(k\). The
	\((4,x)\)-deficit is dominated by the \((3,x)\)-deficit. Put
	\(b_{20}=6\), \(b_{21}=5\), \(e_{20}=1/22\), and \(e_{21}=5/231\).
	For every integer \(0\le k\le t\), the total deficit is at most
	\[
	k\left(\frac{10}{t}
	+\frac6{\max\{b_t,k\}}-1\right)+(t-k)e_t.
	\]
	This expression increases up to \(k=b_t\) and decreases thereafter,
	so its maximum occurs at \(k=b_t\). It yields
	\[
	\begin{array}{c|c|c}
		t & \text{largest total deficit} & \text{retained charge}\\ \hline
		20 & 40/11 & 4\\
		21 & 41/11 & 4
	\end{array}
	\]
	The maxima occur at \(k=6\) and \(k=5\), respectively. Both are
	strictly smaller than \(4\), so every face of modified weight \(20\)
	or \(21\) can pay all incident deficits.
	
	\subsubsection{Large faces of modified weight at least 22}
	
	Finally, let \(f\) be a large face and put \(t=w'(f)\ge22\).
	An exceptional vertex has a charge of at least
	\[
	-3+\left(2-\frac{10}{t}\right)
	+\left(2-\frac6{11}\right)
	=1-\frac{10}{t}-\frac6{11}\ge0.
	\]
	Thus no exceptional vertex is special on \(f\), and every two
	special vertices on \(f\) share another incident face.
	
	Let \(S\) be the set of special vertices on \(f\). We first handle
	\(|S|\le3\), where a numerical estimate alone is not sufficient. A
	high-degree special vertex has deficit at most
	\(1+10/t\le16/11\); equality in this bound can occur only for a
	degree-\(3\) vertex with reduced configuration \((3,3)\). Every
	other high-degree types have a deficit at most
	\(1/2+10/t\le21/22\), except that a degree-\(4\) type has the still
	smaller deficit \(10/t\). There is at most one special
	\(2\)-vertex on \(f\): any two special vertices share a second face,
	whereas a pseudodigon leaves only one incident \(2\)-vertex unpaid
	in (D1).
	
	If \(|S|\le2\), the total deficit is less than \(4\), even when one
	member is a \(2\)-vertex. Suppose \(|S|=3\) and the total deficit
	exceeds \(4\). If \(S\) contains no \(2\)-vertex, all three vertices
	must have reduced configuration \((3,3)\); otherwise the total is at
	most \(2(16/11)+21/22<4\). By Claim~1(v), every pair shares a face
	different from \(f\). Claim~3(a) implies that these three secondary
	faces are distinct. For each pair, its nontrivial common path is the
	arc of the boundary of \(f\) that avoids the third special vertex,
	because every internal vertex of such a path has degree \(2\). The
	three resulting arcs partition the boundary of \(f\), and all their
	internal vertices have degree \(2\). Hence \(w(f)=3\), contrary to
	\(t\ge22\).
	
	Now suppose that \(S\) contains a special \(2\)-vertex \(z\). Its
	other incident face \(g\) is a pseudodigon. Every other member of
	\(S\) must also be incident with \(g\), and hence must be one of the
	two \(3^+\)-vertices on \(g\). If the total deficit exceeds \(4\),
	both of these vertices have degree \(3\); a degree-\(4\) vertex
	would contribute at most \(10/t\), making the total smaller than
	\(4\). All internal vertices of either boundary arc of the pseudodigon
	\(g\) have degree \(2\), so each arc has a single face on its other
	side. One of these faces is \(f\); call the other one \(h\). The two
	degree-\(3\) endvertices are therefore incident with the same three
	faces \(f,g,h\), and hence have
	the same reduced configuration. The total can exceed \(4\) only if
	this configuration is \((3,3)\). Let the two vertices be \(p\) and
	\(q\). The maximal common \(f\)--\(g\) boundary path through \(z\)
	is a \(p\)--\(q\) path, since all other vertices of the pseudodigon
	have degree \(2\). By Claim~3(a), the common \(f\)--\(h\) boundary
	path is the complementary \(p\)--\(q\) arc of the boundary of \(f\),
	and all its internal vertices also have degree \(2\). Thus \(p,q\)
	are the only \(3^+\)-vertices on \(f\), so \(w(f)=2\), again a
	contradiction.
	
	It remains to consider \(k=|S|\ge4\). There is then no special
	\(2\)-vertex: if \(z\) were one, pairwise secondary cofaciality would
	force every other member of \(S\) to lie on the pseudodigon incident
	with \(z\), which has only two high-degree boundary vertices, giving
	\(|S|\le3\). By Claim~3(b), all vertices of \(S\) are incident with
	one common face \(g\ne f\), whose modified weight is at least \(k\).
	Since every entry in every remaining reduced configuration is at most
	\(11\), these observations also imply \(k\le11\).
	For a reduced configuration \((3,x)\), the deficit is at most
	\[
	\frac{10}{t}+\frac6k-1
	<\frac6k-\frac12.
	\]
	The other reduced configurations have smaller deficits, and
	\((3,3,3)\) cannot contain a common face of weight at least \(4\).
	Consequently, the total deficit is less than
	\(k(6/k-1/2)=6-k/2\le4\). Therefore, every large face of modified
	weight at least \(22\) can pay all incident deficits.
	
	\subsubsection{Completion of the discharging argument}
	
	Apply one final rule:
	\begin{enumerate}
		\item[(D3)] Every non-small face gives each incident special
		vertex an amount equal to that vertex's deficit.
	\end{enumerate}
	Every special vertex has a unique incident non-small face, so it is
	paid exactly once. The preceding three subsections show that each
	medium or large face retained enough charge to make all these
	payments. Small faces finish with charge \(0\), and all vertices and
	all other faces finish with nonnegative charge. This contradicts
	\eqref{eq:total-charge}. Therefore, no counterexample exists, and
	\(\chi_{\mathrm{pfo}}(\mathcal P)\le92\).
\end{proof}

\subsection{Consequences for planar graphs}

\begin{theorem}
	\label{thm:strong-odd-planar-368}
	Every planar graph \(G\) satisfies \(\chi_{\mathrm{so}}(G)\le 368\).
\end{theorem}

\begin{proof}
	By the Four Color Theorem, every planar graph \(G\) satisfies \(\chi(G)\le 4\). By the reduction of Caro et al., we have \(\chi_{\mathrm{so}}(G)\le \chi(G)\chi_{\mathrm{pfo}}(\mathcal P)\). By Theorem~\ref{thm:pfo-92}, this gives
	\(
	\chi_{\mathrm{so}}(G)\le 4\cdot 92=368.
	\)
	This completes the proof.
\end{proof}

\begin{corollary}
	There exists a constant \(c<388\) such that every planar graph \(G\) satisfies \(\chi_{\mathrm{so}}(G)<c\). In particular, one may take \(c=369\).
\end{corollary}

%%%%%%%%%%%%%%%%%%%%%%%%%%%%

\section{A planar graph with strong odd chromatic number 20}
\label{sec:planar-strong-odd-20}

Manattu et al.~\cite{ManattuVinayLakshmanan2026}
constructed planar graphs with strong odd chromatic number up to \(17\)
and asked for planar graphs with a larger strong odd chromatic number.
Chronologically, the existence of a planar graph with a strong odd
chromatic number \(20\) was already stated in the arXiv preprint of
Goetze et al.~\cite{GoetzeKluteKnauerParadaPenaUeckerdt2025}, posted in
May 2025. Their preprint exhibits such a graph, but does not provide a
detailed verification of the exact value. In this section, we give a
different explicit planar graph
\(G\) and provide a complete proof that \(\chi_{\mathrm{so}}(G)=20\).
Thus, this section should be read as a self-contained, fully verified
construction in the direction of Problem~1 of Manattu et al., not as a
priority claim for the value \(20\).

\subsection{Construction of the graph}

Let \(X=\{x_1,x_2,x_3\}\). Let
\(P_{12}=v_{12}^{1}v_{12}^{2}\cdots v_{12}^{6}\),
\(P_{13}=v_{13}^{1}v_{13}^{2}\cdots v_{13}^{6}\), and
\(P_{23}=v_{23}^{1}v_{23}^{2}\cdots v_{23}^{6}\) be three pairwise
vertex-disjoint paths on six vertices. Assume that these three paths are
also disjoint from \(X\).

The graph \(G\) has vertex set
\(
V(G)=X\cup V(P_{12})\cup V(P_{13})\cup V(P_{23}).
\)
Its edge set is defined as follows. First, keep all edges of the three
paths \(P_{12},P_{13},P_{23}\). Next, join \(x_1\) and \(x_2\) to every
vertex of \(P_{12}\), join \(x_1\) and \(x_3\) to every vertex of
\(P_{13}\), and join \(x_2\) and \(x_3\) to every vertex of \(P_{23}\).
Finally, add the edge \(x_1x_2\). There are no other edges. Thus,
\(x_1x_2\) is the only edge in the subgraph induced by \(X\).
Figure~\ref{fig:G20} illustrates the construction.

\begin{figure*}[t]
	\centering
	\resizebox{0.90\textwidth}{!}{%
		\begin{tikzpicture}[
			bigv/.style={
				circle,
				draw,
				fill=white,
				inner sep=1pt,
				minimum size=7mm,
				font=\small
			},
			smallv/.style={
				circle,
				draw,
				fill=white,
				inner sep=0.4pt,
				minimum size=4.5mm,
				font=\scriptsize
			},
			pathedge/.style={
				line width=0.5pt
			},
			joinedge/.style={
				line width=0.28pt,
				gray!75
			},
			hubedge/.style={
				line width=0.8pt
			},
			every node/.style={font=\small}
			]
			
			% Main vertices
			\node[bigv] (x1) at (0,4.1) {$x_1$};
			\node[bigv] (x2) at (-4.2,-0.9) {$x_2$};
			\node[bigv] (x3) at (4.2,-0.9) {$x_3$};
			
			% The additional edge x_1x_2
			\draw[hubedge]
			(x1) .. controls (-4.8,4.4) and (-5.3,0.2) .. (x2);
			
			% Path P12
			\node[smallv] (a1) at (-1.45,2.95) {$1$};
			\node[smallv] (a2) at (-1.55,2.35) {$2$};
			\node[smallv] (a3) at (-1.65,1.75) {$3$};
			\node[smallv] (a4) at (-1.75,1.15) {$4$};
			\node[smallv] (a5) at (-1.85,0.55) {$5$};
			\node[smallv] (a6) at (-1.95,-0.05) {$6$};
			\node at (-2.25,2.95) {$P_{12}$};
			
			% Path P13
			\node[smallv] (b1) at (1.45,2.95) {$1$};
			\node[smallv] (b2) at (1.55,2.35) {$2$};
			\node[smallv] (b3) at (1.65,1.75) {$3$};
			\node[smallv] (b4) at (1.75,1.15) {$4$};
			\node[smallv] (b5) at (1.85,0.55) {$5$};
			\node[smallv] (b6) at (1.95,-0.05) {$6$};
			\node at (2.25,2.95) {$P_{13}$};
			
			% Path P23
			\node[smallv] (c1) at (-0.65,-1.55) {$1$};
			\node[smallv] (c2) at (-0.39,-2.05) {$2$};
			\node[smallv] (c3) at (-0.13,-2.55) {$3$};
			\node[smallv] (c4) at (0.13,-3.05) {$4$};
			\node[smallv] (c5) at (0.39,-3.55) {$5$};
			\node[smallv] (c6) at (0.65,-4.05) {$6$};
			\node at (0.00,-1.05) {$P_{23}$};
			
			% Edges of the three paths
			\draw[pathedge] (a1)--(a2)--(a3)--(a4)--(a5)--(a6);
			\draw[pathedge] (b1)--(b2)--(b3)--(b4)--(b5)--(b6);
			\draw[pathedge] (c1)--(c2)--(c3)--(c4)--(c5)--(c6);
			
			% Join edges from x1 and x2 to P12
			\foreach \name in {a1,a2,a3,a4,a5,a6}{
				\draw[joinedge]
				(x1) .. controls ($(x1)!0.45!(\name)+(0.65,0.20)$)
				and ($(\name)+(0.42,0.05)$) .. (\name);
				\draw[joinedge]
				(x2) .. controls ($(x2)!0.45!(\name)+(-0.70,-0.10)$)
				and ($(\name)+(-0.42,-0.05)$) .. (\name);
			}
			
			% Join edges from x1 and x3 to P13
			\foreach \name in {b1,b2,b3,b4,b5,b6}{
				\draw[joinedge]
				(x1) .. controls ($(x1)!0.45!(\name)+(-0.65,0.20)$)
				and ($(\name)+(-0.42,0.05)$) .. (\name);
				\draw[joinedge]
				(x3) .. controls ($(x3)!0.45!(\name)+(0.70,-0.10)$)
				and ($(\name)+(0.42,-0.05)$) .. (\name);
			}
			
			% Join edges from x2 and x3 to P23
			\foreach \name in {c1,c2,c3,c4,c5,c6}{
				\draw[joinedge]
				(x2) .. controls ($(x2)!0.45!(\name)+(-0.75,-0.20)$)
				and ($(\name)+(-0.45,0)$) .. (\name);
				\draw[joinedge]
				(x3) .. controls ($(x3)!0.45!(\name)+(0.75,-0.20)$)
				and ($(\name)+(0.45,0)$) .. (\name);
			}
			
		\end{tikzpicture}%
	}
	\caption{A full-edge illustration of the planar graph \(G\). The six
		vertices labeled \(1,\ldots,6\) on \(P_{12}\), \(P_{13}\), and
		\(P_{23}\) represent \(v_{12}^{1},\ldots,v_{12}^{6}\),
		\(v_{13}^{1},\ldots,v_{13}^{6}\), and
		\(v_{23}^{1},\ldots,v_{23}^{6}\), respectively. The bold edge is the
		additional edge \(x_1x_2\).}
	\label{fig:G20}
\end{figure*}
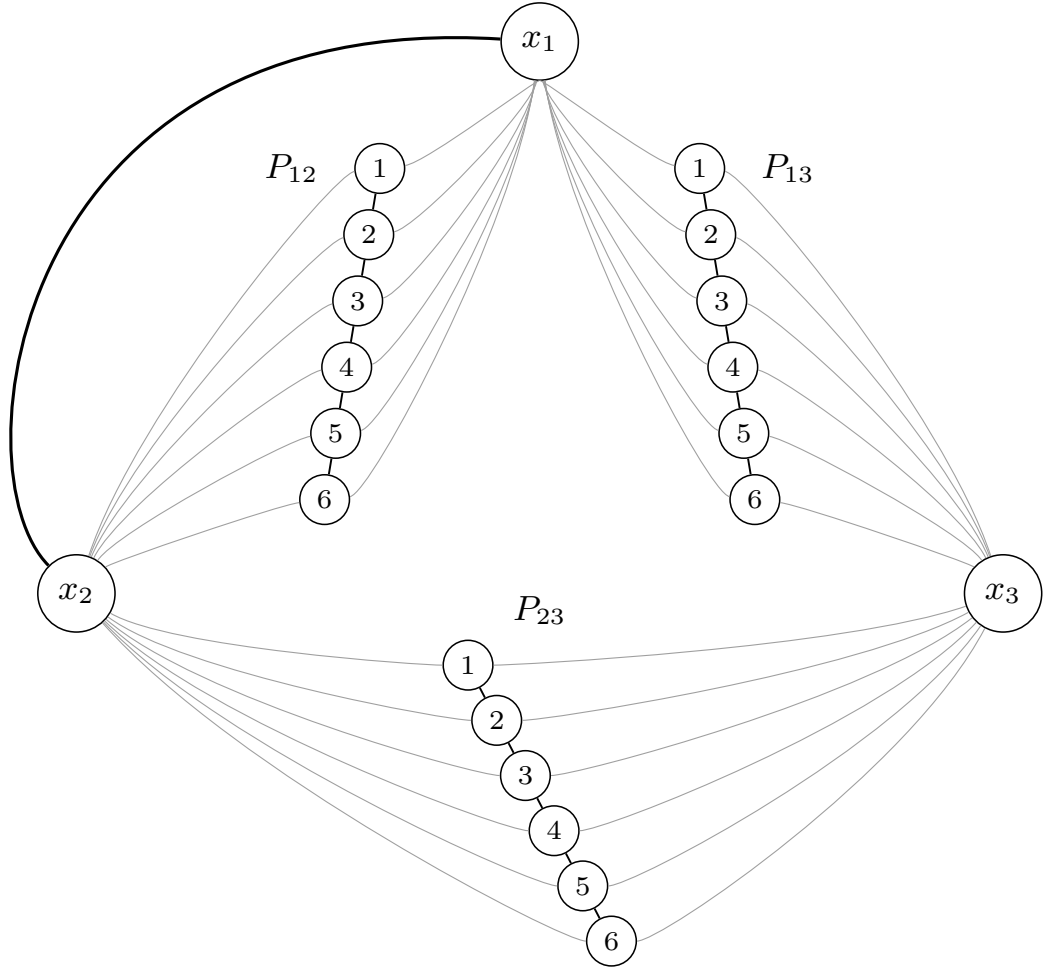

\begin{lemma}
	\label{lem:G20-planar}
	The graph \(G\) is planar.
\end{lemma}

\begin{proof}
	First omit the edge \(x_1x_2\). The subgraph induced by
	\(\{x_1,x_2\}\cup V(P_{12})\), the subgraph induced by
	\(\{x_1,x_3\}\cup V(P_{13})\), and the subgraph induced by
	\(\{x_2,x_3\}\cup V(P_{23})\) are each isomorphic to
	\(P_6\vee\overline{K_2}\). Each of these graphs has a plane embedding in
	which the two vertices outside the path lie on a common face: place the
	six vertices of the path in order on a curve, place the two remaining
	vertices on opposite sides of the curve, and draw the two fans to the
	path vertices on the two sides.
	
	Choose three closed discs with pairwise disjoint interiors, one for
	each of the pairs \(\{x_1,x_2\}\), \(\{x_1,x_3\}\), and
	\(\{x_2,x_3\}\), such that the two corresponding vertices lie on the
	boundary of each disc. Draw the appropriate copy of
	\(P_6\vee\overline{K_2}\) inside each disc. The drawings meet only at
	the hubs \(x_1,x_2,x_3\). The boundary arcs of the first disc leave
	one \(x_1\)-to-\(x_2\) arc unused; draw the additional edge
	\(x_1x_2\) along this arc. No crossing is created, so \(G\) is
	planar.
\end{proof}

\subsection{Colors on an individual path}

Let \(\varphi\) be a strong odd coloring of \(G\). For a color
\(\alpha\), let \(V_\alpha=\{v\in V(G):\varphi(v)=\alpha\}\). Define
\[
\begin{aligned}
	m_{12}(\alpha)&=|V_\alpha\cap V(P_{12})|,\\
	m_{13}(\alpha)&=|V_\alpha\cap V(P_{13})|,\\
	m_{23}(\alpha)&=|V_\alpha\cap V(P_{23})|.
\end{aligned}
\]

\begin{lemma}
	\label{lem:G20-path-multiplicity}
	For every color \(\alpha\), each of
	\(m_{12}(\alpha)\), \(m_{13}(\alpha)\), and \(m_{23}(\alpha)\) is at most
	\(2\).
\end{lemma}

\begin{proof}
	We prove the claim for \(P_{12}\); the arguments for \(P_{13}\) and
	\(P_{23}\) are identical. Two consecutive vertices of \(P_{12}\) cannot
	receive the same color because \(\varphi\) is proper. Two vertices at
	distance two on \(P_{12}\) cannot receive the same color either. Indeed,
	their common neighbour \(P_{12}\) would then see that color exactly
	twice. Moreover, neither \(x_1\) nor \(x_2\) can have that same color,
	because both \(x_1\) and \(x_2\) are adjacent to every vertex of
	\(P_{12}\). Thus, the color would appear exactly twice in the neighbourhood
	of the common neighbour, contradicting the strong odd condition.
	
	Consequently, vertices of one color on \(P_{12}\) are pairwise at
	distance at least three. A path on six vertices contains at most two such
	vertices. Therefore, \(m_{12}(\alpha)\le 2\). The same proof gives
	\(m_{13}(\alpha)\le 2\) and \(m_{23}(\alpha)\le 2\).
\end{proof}

\subsection{The path vertices form a rainbow set}

\begin{lemma}
	\label{lem:G20-rainbow}
	In every strong odd coloring of \(G\), the \(18\) vertices in
	\(V(P_{12})\cup V(P_{13})\cup V(P_{23})\) receive pairwise distinct
	colors. Moreover, neither \(\varphi(x_1)\) nor \(\varphi(x_2)\) appears
	on a path vertex.
\end{lemma}

\begin{proof}
	First, observe that \(x_1,x_2,x_3\) receive pairwise distinct colors.
	The vertices \(x_1\) and \(x_2\) are adjacent, so
	\(\varphi(x_1)\neq\varphi(x_2)\). If
	\(\varphi(x_1)=\varphi(x_3)\), then every vertex of \(P_{13}\) sees this
	color exactly twice, once on \(x_1\) and once on \(x_3\), which
	contradicts the strong odd condition. Similarly,
	\(\varphi(x_2)\neq\varphi(x_3)\).
	
	Fix a color \(\alpha\), and write
	\(m_{12}=m_{12}(\alpha)\), \(m_{13}=m_{13}(\alpha)\), and
	\(m_{23}=m_{23}(\alpha)\). Suppose first that
	\(\alpha\notin\{\varphi(x_1),\varphi(x_2),\varphi(x_3)\}\). At
	\(x_1,x_2,x_3\), respectively, the numbers of neighbors colored
	\(\alpha\) are \(m_{12}+m_{13}\), \(m_{12}+m_{23}\), and
	\(m_{13}+m_{23}\). Therefore, each of these three sums is either zero or
	odd.
	
	We claim that at most one of \(m_{12},m_{13},m_{23}\) is positive.
	Suppose exactly two are positive, say \(m_{12}>0\), \(m_{13}>0\), and
	\(m_{23}=0\). The condition at \(x_2\) implies that \(m_{12}\) is odd,
	and the condition at \(x_3\) implies that \(m_{13}\) is odd. Hence,
	\(m_{12}+m_{13}\) is positive and even, contradicting the condition at
	\(x_1\). The other choices of two positive values are identical.
	
	If all three values are positive, then the strong odd conditions at
	\(x_1,x_2,x_3\) give
	\(
	m_{12}+m_{13}\equiv
	m_{12}+m_{23}\equiv
	m_{13}+m_{23}\equiv 1\pmod 2.
	\)
	Adding these three congruences gives \(0\equiv 1\pmod 2\), which is
	impossible. Thus exactly one of \(m_{12},m_{13},m_{23}\) is positive
	whenever \(\alpha\) appears on a path. The strong odd condition at the
	two vertices adjacent to that path implies that the positive value is
	odd. By Lemma~\ref{lem:G20-path-multiplicity}, this positive value is
	equal to \(1\).
	
	It remains to consider colors used on \(X\). Let
	\(\alpha=\varphi(x_1)\). By properness, \(\alpha\) does not appear on
	\(P_{12}\) or \(P_{13}\), and hence it can appear only on \(P_{23}\). If
	\(m_{23}>0\), then \(x_3\) sees \(\alpha\) exactly \(m_{23}\) times, so
	\(m_{23}\) must be odd. On the other hand, \(x_2\) sees \(\alpha\) on
	\(x_1\) and on the \(m_{23}\) vertices of \(P_{23}\). Therefore,
	\(1+m_{23}\) must be odd, so \(m_{23}\) must be even, a contradiction.
	Hence, \(\varphi(x_1)\) does not appear on any path.
	
	The same argument, with the roles of \(x_1\) and \(x_2\) interchanged,
	shows that \(\varphi(x_2)\) does not appear on any path. Finally, let
	\(\alpha=\varphi(x_3)\). By properness, \(\alpha\) does not appear on
	\(P_{13}\) or \(P_{23}\), and it may appear only on \(P_{12}\). If it
	appears there, then the strong odd condition at both \(x_1\) and \(x_2\)
	implies that \(m_{12}\) is odd. By
	Lemma~\ref{lem:G20-path-multiplicity}, we get \(m_{12}=1\).
	
	It follows that every color appears on at most one of the \(18\) path
	vertices. Thus the path vertices receive pairwise distinct colors, and
	the colors of \(x_1\) and \(x_2\) are absent from all three paths.
\end{proof}

\subsection{The lower bound}

\begin{lemma}
	\label{lem:G20-lower}
	Every strong odd coloring of \(G\) uses at least \(20\) colors.
\end{lemma}

\begin{proof}
	By Lemma~\ref{lem:G20-rainbow}, the \(18\) path vertices receive
	pairwise distinct colors. The colors of \(x_1\) and \(x_2\) do not
	appear on any path vertex, and they are distinct because \(x_1x_2\in
	E(G)\). Therefore \(x_1\) and \(x_2\) require two additional colors. The
	color of \(x_3\) may coincide with the color of one vertex of \(P_{12}\),
	but this cannot reduce the number of colors below \(18+2=20\). Hence,
	\(\chi_{\mathrm{so}}(G)\ge 20\).
\end{proof}

\subsection{A strong odd 20-coloring}

\begin{lemma}
	\label{lem:G20-upper}
	The graph \(G\) admits a strong odd coloring using \(20\) colors.
\end{lemma}

\begin{proof}
	Assign pairwise distinct colors to the \(18\) path vertices. Let
	\(c_{12}^{r}\), \(c_{13}^{r}\), and \(c_{23}^{r}\) be the colors assigned
	to \(v_{12}^{r}\), \(v_{13}^{r}\), and \(v_{23}^{r}\), respectively, for
	\(r=1,2,\ldots,6\). Let \(a\) and \(b\) be two new colors not used on any
	path vertex, and define \(\varphi(x_1)=a\), \(\varphi(x_2)=b\), and
	\(\varphi(x_3)=c_{12}^{1}\). This coloring uses exactly \(20\) colors.
	
	The coloring is proper. Indeed, \(x_1\) and \(x_2\) receive two distinct
	new colors, and \(x_3\) receives a color used on \(P_{12}\), whose
	vertices are not adjacent to \(x_3\).
	
	Every color appearing in \(N_G(x_1)\) appears exactly once, since
	\(N_G(x_1)=\{x_2\}\cup V(P_{12})\cup V(P_{13})\), and these vertices
	have pairwise distinct colors. The same argument applies to \(x_2\).
	Also, \(N_G(x_3)=V(P_{13})\cup V(P_{23})\), and all vertices in this
	neighbourhood have pairwise distinct colors.
	
	Now consider a vertex \(v\in V(P_{12})\). Its neighbourhood consists of
	\(x_1,x_2\) and one or two neighbouring vertices on \(P_{12}\). These
	neighbours have pairwise distinct colors. A vertex of \(P_{13}\) has as
	neighbours \(x_1,x_3\) and one or two neighboring vertices on \(P_{13}\);
	again, all these colors are distinct. A vertex of \(P_{23}\) has as
	neighbours \(x_2,x_3\) and one or two neighbouring vertices on \(P_{23}\),
	and these colors are also distinct. Therefore, every color present in
	every open neighbourhood appears exactly once. Hence, \(\varphi\) is a
	strong odd \(20\)-coloring of \(G\), and so
	\(\chi_{\mathrm{so}}(G)\le 20\).
\end{proof}

\begin{theorem}
	\label{thm:planar-strong-odd-20}
	There exists a planar graph \(G\) such that
	\(
	\chi_{\mathrm{so}}(G)=20.
	\)
\end{theorem}

\begin{proof}
	The graph \(G\) is planar by Lemma~\ref{lem:G20-planar}. By
	Lemma~\ref{lem:G20-lower}, every strong odd coloring of \(G\) uses at
	least \(20\) colors. By Lemma~\ref{lem:G20-upper}, the graph \(G\)
	admits a strong odd coloring using \(20\) colors. Therefore,
	\(\chi_{\mathrm{so}}(G)=20\).
\end{proof}

Theorem~\ref{thm:planar-strong-odd-20} gives a self-contained explicit
construction of a planar graph with strong odd chromatic number greater
than \(17\). This provides a fully verified construction in the direction
of Problem~1 of Manattu et al.~\cite{ManattuVinayLakshmanan2026}. We emphasize again that
the value \(20\) had already appeared in the earlier arXiv preprint of
Goetze et al.~\cite{GoetzeKluteKnauerParadaPenaUeckerdt2025}; our
contribution here is a different construction together with a complete
proof.

\section{Conclusion}
\label{sec:conclusion}

In this paper, we studied problems raised by Manattu et
al.~\cite{ManattuVinayLakshmanan2026} on the strong odd chromatic number
of planar graphs. Our main contribution is an improved universal upper
bound. The previous bound \(\chi_{\mathrm{so}}(G)\le 388\) follows from
the reduction \(\chi_{\mathrm{so}}(G)\le
\chi(G)\chi_{\mathrm{pfo}}(\mathcal P)\), the Four Color Theorem, and
the facially odd coloring bound \(\chi_{\mathrm{pfo}}(\mathcal P)\le 97\).
We improved the auxiliary facially odd coloring bound to
\(\chi_{\mathrm{pfo}}(\mathcal P)\le 92\). Consequently, every planar
graph \(G\) satisfies \(\chi_{\mathrm{so}}(G)\le 368\). Therefore, in the
form of Problem~2 asking for a constant \(c<388\) such that
\(\chi_{\mathrm{so}}(G)<c\) for every planar graph \(G\), one may take
\(c=369\).

We also presented an explicit planar graph \(G\) with
\(\chi_{\mathrm{so}}(G)=20\), together with a complete proof of the exact
value. This gives a direct and self-contained construction whose value
exceeds \(17\), the value appearing in Problem~1 of Manattu et
al.~\cite{ManattuVinayLakshmanan2026}. However, to avoid any ambiguity
about priority, we recall that Goetze et
al.~\cite{GoetzeKluteKnauerParadaPenaUeckerdt2025} had already posted an
arXiv preprint in May 2025 containing a planar example with a strong odd
chromatic number \(20\). Thus, our lower-bound section should be viewed as
a different explicit construction with full verification, rather than the
first proof of the existence of such a graph.

\begin{remark}
	Manattu et al.~\cite{ManattuVinayLakshmanan2026} also asked whether
	there exists a constant \(c\) such that only finitely many planar graphs
	\(G\) satisfy \(\chi_{\mathrm{so}}(G)>c\). The previously known universal
	bound \(388\) had already answered this question affirmatively by allowing
	\(c=388\); our result merely improves the admissible constant to \(368\).
	Moreover, since the problem is stated for planar graphs and not only for
	connected planar graphs, one should note that if a planar graph \(G\) with
	\(\chi_{\mathrm{so}}(G)>c\) exists, then the graphs \(G\cup qK_1\), where
	\(q\) is any integer with \(q\ge 0\), give infinitely many planar graphs
	with the same strong odd chromatic number. Indeed, adding isolated
	vertices does not change the strong odd chromatic number. Thus, in the
	class of all planar graphs, the finiteness question is essentially
	equivalent to finding a universal upper bound.
\end{remark}

The gap between the lower bound \(20\) and the upper bound \(368\)
remains large. We believe that the lower bound gives the correct extremal
value, and we therefore propose the following conjecture.

\begin{conjecture}
	Every planar graph \(G\) satisfies \(\chi_{\mathrm{so}}(G)\le 20\).
\end{conjecture}

Equivalently, the maximum value of \(\chi_{\mathrm{so}}(G)\) over all planar
graphs is \(20\). Proving this conjecture, or finding a planar graph with
strong odd chromatic number at least \(21\), remains the main open direction.
It would also be interesting to determine whether the facially odd coloring
approach can be refined to reduce the general upper bound substantially.

\bibliographystyle{plain}
%\addcontentsline{toc}{section}{Bibliography}
%\bibliographystyle{splncs04}
\bibliography{SOC_ref}

\end{document}